\documentclass{amsart}
\usepackage{amssymb}
\usepackage{amsmath}
\usepackage{amsfonts}

\newtheorem{theorem}{Theorem}
\theoremstyle{plain}

\newtheorem{definition}{Definition}
\newtheorem{example}{Example}

\newtheorem{solution}{Solution}

\numberwithin{equation}{section}
\input{tcilatex}

\begin{document}
\title[Completion of Quasi 2-normed space ]{A study involving the Completion
of Quasi 2-normed space}
\author{Mehmet K\i r}
\address{Atat\"{u}rk University, Faculty of Science, Department of
Mathematics, 25000 Erzurum, TURKEY}
\email{mehmet\_040465@yahoo.com}
\author{Mehmet Acikgoz}
\address{University of Gaziantep, Faculty of Science and Arts, Department of
Mathematics, 27310 Gaziantep, TURKEY}
\email{acikgoz@gantep.edu.tr}
\date{January 5, 2012}

\begin{abstract}
The fundamental aim of this paper is to introduce and investigate a new
property of quasi $2$-normed space based on a question given by C. Park
(2006) [2] for the completion quasi $2$-normed space. Finally, we also find
an answer for a question Park's.
\end{abstract}

\maketitle

\section{Introduction, Definitions and Notations}

\bigskip In 1928, K. Menger introduced the notion called n-metrics (or
generalized metric). But many mathematicians had not paid attentions to
Menger's theory about generalized metrics. But{} several mathematicians, A.
Wald, L. M. Blumenthal, W. A. Wilson etc. have developed Menger's idea.

In 1963, S. G\"{a}hler limits Menger's considerations to $n=2$. G\"{a}hler's
study is more complete in view of the fact that he developes the topological
properties of the spaces in question. G\"{a}hler also proves that if the
space is a linear normed space, then it is possible to define 2-norm.

Since 1963, S. G\"{a}hler, Y. J. Cho, R. W. Frees, C. R. Diminnie, R. E.
Ehret, K. Is\'{e}ki, A. White and many others have studied on 2-normed
spaces and 2-metric spaces. It is well-known that $%
\mathbb{R}
$ is complete but $%
\mathbb{Q}
$ is not complete. Since $%
\mathbb{Q}
$ is dense in $%
\mathbb{R}
$ it is said that R is completion of $%
\mathbb{Q}
$. It is very important that an incomplete space can be completed in similar
sense. Complete spaces, in other words Banach spaces, play quite important
role in many branches of mathematics and its applications. Many
mathematicians showed the existince of completion of normed spaces (for more
information, see [1], [2], [6]). We shall also show completion of quasi
2-normed spaces via similar sense.

\begin{definition}
Let $X$ be a real linear space with $dim\geq 2$ and $\left\Vert
.,.\right\Vert :X^{2}\rightarrow \lbrack 0,\infty )$ a function. Then $%
\left( X,\left\Vert .,.\right\Vert \right) $ is called linear 2-normed
spaces if
\end{definition}

\QTP{Body Math}
$2N_{1})\left\Vert x,y\right\Vert =0\Longleftrightarrow $ $x$ \textit{and} $%
y $ \textit{linearly dependent,}

\QTP{Body Math}
$2N_{2})\left\Vert x,y\right\Vert =\left\Vert y,x\right\Vert ,$

\QTP{Body Math}
$2N_{3})\left\Vert \alpha x,y\right\Vert =\left\vert \alpha \right\vert
\left\Vert x,y\right\Vert ,$

\QTP{Body Math}
$2N_{4})\left\Vert x+y,z\right\Vert =\left\Vert x,z\right\Vert +\left\Vert
y,z\right\Vert ,$

\QTP{Body Math}
\textit{for all} $\alpha \in 
\mathbb{R}
$ \textit{and all} $x,$ $y,$ $z\in X.$

\begin{example}
Let $E_{3}$ denotes Euclidean vector three spaces. Let $x=ai+bj+ck$ and $%
y=di+ej+fk$ define%
\begin{eqnarray*}
\left\Vert x,y\right\Vert  &=&\left\vert x\times y\right\vert =abs\left\vert 
\begin{array}{ccc}
i & j & k \\ 
a & b & c \\ 
d & e & f%
\end{array}%
\right\vert  \\
&=&\left\vert \left( bf-ce\right) ^{2}i+\left( cd-af\right) ^{2}j+\left(
ae-db\right) ^{2}k\right\vert ^{\frac{1}{2}}
\end{eqnarray*}%
Then $\left( E_{3},\left\Vert .,.\right\Vert \right) $ is a 2-normed space
and this space is complete (for more information, see [6]).
\end{example}

Also

1) In addition to $2N_{1}),$ $2N_{2}),$ $2N_{3}),$ if there is a constant $%
K\geq 1$ such that 
\begin{equation*}
2N_{4}^{\ast })\text{ }\left\Vert x+y,z\right\Vert \leq K\left( \left\Vert
x,z\right\Vert +\left\Vert y,z\right\Vert \right) \text{ for all }x,\text{ }%
y,\text{ }z\in X
\end{equation*}

is called quasi 2-normed space.

2) A $2$-norm $\left\Vert .,.\right\Vert $ defined on a linear space $X$ is
said to be uniformly continuous in both variables if for any $\varepsilon >0$
there exist a neighbourhood $U_{\varepsilon }$ of $0$ such that $\left\vert
\left\Vert a,b\right\Vert -\left\Vert a%
{\acute{}}%
,b%
{\acute{}}%
\right\Vert \right\vert <\varepsilon $ whenever $a$-$a%
{\acute{}}%
$ and $b$-$b%
{\acute{}}%
$ are in $U_{\varepsilon },$ which is independent of the choice of $a$, $a%
{\acute{}}%
$, $b$, $b%
{\acute{}}%
.$

3) A pseudo $2$-norm is defined to be real-valued function having all the
properties of $2$-norm $\left\Vert .,.\right\Vert $ except the condition
that $\left\Vert a,b\right\Vert =0$ implies the linear dependence of $a$ and 
$b$ (for details, see [1]).

\begin{example}
Let $X$ be a linear space with $\dim X\geq 2$ and $\left\Vert .,.\right\Vert 
$ be $2$-norm on $X.$ 
\begin{equation*}
\left\Vert x,y\right\Vert _{q}=2\left\Vert x,y\right\Vert 
\end{equation*}%
is quasi $2$-norm on $X$ and $\left( X,\left\Vert .,.\right\Vert _{q}\right) 
$ is quasi $2$-normed space.
\end{example}

\begin{solution}
By using conditions $2N_{1}),$ $2N_{2}),$ $2N_{3})$ in $2$-normed spaces,
however, using $2N_{4}^{\ast })$, we show that as follows: 
\begin{equation*}
2N_{1})\left\Vert x,y\right\Vert _{q}=0\text{ if and only if }2\left\Vert
x,y\right\Vert =0,\text{ namely, }x\text{ and }y\text{ linearly dependent,}
\end{equation*}%
It is easy to see for $2N_{2}),$ that \noindent is, $\left\Vert
x,y\right\Vert _{q}=\left\Vert y,x\right\Vert _{q},$ Using property of $%
2N_{3}),$ we readily see the following applications: For all $\alpha \in 
\mathbb{R}
,$%
\begin{eqnarray*}
\left\Vert \alpha x,y\right\Vert _{q} &=&2\left\Vert \alpha x,y\right\Vert
=\left\vert \alpha \right\vert \left( 2\left\Vert x,y\right\Vert \right)  \\
&=&\left\vert \alpha \right\vert \left\Vert x,y\right\Vert _{q}
\end{eqnarray*}%
We are now ready to prove property of $2N_{4}^{\ast }),$ That is, For all $x$%
, $y$, $z\in X$%
\begin{eqnarray*}
\left\Vert x+y,z\right\Vert _{q} &=&2\left\Vert x+y,z\right\Vert  \\
&\leq &2\left( \left\Vert x,z\right\Vert +\left\Vert y,z\right\Vert \right) 
\\
&=&\left\Vert x,z\right\Vert _{q}+\left\Vert y,z\right\Vert _{q}
\end{eqnarray*}%
So, we complete solution of Example $2$.
\end{solution}

\begin{theorem}
Let $X$ be a linear space with $\dim X\geq 2$ and $\left\Vert .,.\right\Vert 
$ be $2$-norm on $X,$ for constant $a$, $b\in 
\mathbb{R}
$ which are $a\geq \frac{1}{2}$ and $b\geq \frac{1}{2}.$ There exists $%
\left\Vert x,y\right\Vert _{q}$ quasi $2$-norm on $X$ defined as%
\begin{equation*}
\left\Vert x,y\right\Vert _{q}=a\left\Vert x,y\right\Vert +b\left\Vert
x,y\right\Vert 
\end{equation*}
\end{theorem}

\begin{proof}
It is evident to show conditions $2N_{1}),$ $2N_{2})$ and $2N_{3}),$
Therefore, It is sufficient to prove condition of $2N_{4}^{\ast })$ as
follows: For all $x$, $y$, $z\in X$%
\begin{eqnarray*}
\left\Vert x+y,z\right\Vert _{q} &=&a\left\Vert x+y,z\right\Vert
+b\left\Vert x+y,z\right\Vert  \\
&\leq &a\left( \left\Vert x,z\right\Vert +\left\Vert y,z\right\Vert \right)
+b\left( \left\Vert x,z\right\Vert +\left\Vert y,z\right\Vert \right)  \\
&=&\left( a+b\right) \left\Vert x,z\right\Vert +\left( a+b\right) \left\Vert
y,z\right\Vert  \\
&=&K\left( \left\Vert x,z\right\Vert +\left\Vert y,z\right\Vert \right) 
\end{eqnarray*}

since there exists a constant $K\geq 1,$ namely, by substituting $K:=a+b,$
we show that $\left\Vert .,.\right\Vert _{q}$ is a quasi $2$-norm on $X.$
\end{proof}

\begin{definition}
Let $\left( X,\left\Vert .,.\right\Vert \right) $ be a quasi $2$-normed
space.
\end{definition}

\textit{a) A sequence} $\left\{ x_{n}\right\} $\textit{\ is a Cauchy
sequence in a linear quasi }$2$\textit{-normed space} $\left( X,\left\Vert
.,.\right\Vert \right) $\textit{\ if and only if} $\lim_{n,m\rightarrow
\infty }\left\Vert x_{n}-x_{m},z\right\Vert =0$ \textit{for every} $z$ 
\textit{in} $X.$

\textit{b)} \textit{A sequence} $\left\{ x_{n}\right\} $ \textit{in} $X$ 
\textit{is called a convergent sequence if there is an} $x\in X$ \textit{%
such that} $\lim_{n,m\rightarrow \infty }\left\Vert x_{n}-x,z\right\Vert =0$ 
\textit{for every} $z$\textit{\ in} $X.$

\textit{c)} \textit{A\ quasi} $2$-\textit{normed space in which every Cauchy
sequence converges is called complete}.

\begin{definition}
Let $\left( X,\left\Vert .,.\right\Vert _{X}\right) $ and $\left(
Y,\left\Vert .,.\right\Vert _{Y}\right) $ be quasi $2$-normed spaces.
\end{definition}

a%
\'{}%
)\textit{\ A mapping }$T:X\rightarrow Y$ \textit{is said to be isometric or
isometry if for all} $x,$ $y\in X$ 
\begin{equation*}
\left\Vert Tx,Ty\right\Vert _{Y}=\left\Vert x,y\right\Vert _{X}
\end{equation*}%
\ \ \ \ \ b%
\'{}%
) \textit{The space} $X$ \textit{is said to be isometry with the space }$Y$ 
\textit{if there exists a bijective isometry of }$X$ \textit{onto} $Y.$ 
\textit{The spaces} $X$ \textit{and} $Y$ \textit{are called isometric spaces.%
}

\begin{theorem}
If a sequence $\left\{ x_{n}\right\} $ is a Cauchy sequence in a linear $2$%
-normed space $\left( X,\left\Vert .,.\right\Vert \right) ,$ then $%
\lim_{n\rightarrow \infty }\left\Vert x_{n},z\right\Vert $ exists for every $%
z$ in $X$ (for proof, see [1]).
\end{theorem}

\begin{theorem}
If $X$ is a linear space having a uniformly continuous $2$-norm $\left\Vert
.,.\right\Vert $ defined on it, then for any two Cauchy sequences $\left\{
x_{n}\right\} $ and $\left\{ y_{n}\right\} $ in $X,$ $\lim_{n\rightarrow
\infty }\left\Vert x_{n},y_{n}\right\Vert $ (for proof, see [1]).
\end{theorem}

\begin{definition}
Two Cauchy sequences $\left\{ x_{n}\right\} $ and $\left\{ y_{n}\right\} $
in a linear $2$-normed space $\left( X,\left\Vert .,.\right\Vert \right) $
are said to be equivalent, denoted by $\left\{ x_{n}\right\} \thicksim
\left\{ y_{n}\right\} ,$ if for every neighbourhood $U$ of $0$ there is an
integer $N\left( U\right) $ such that $n\geq N\left( U\right) $ implies that 
$x_{n}-y_{n}\in U$ cf. [1].
\end{definition}

\section{COMPLETION OF QUASI 2-NORMED SPACE}

In [2], C. Park introduced quasi $2$-normed spaces and gave some results on $%
p$-normed spaces. Also he introduced a question which was "Construct a
completion of a quasi -2-norm". In this section, we give an answer to this
question.

\begin{theorem}
The relation $\thicksim $ on the set of Cauchy sequences in $X$ is an
equivalence relation on $X.$
\end{theorem}

\begin{proof}
It is clear that $\left\{ x_{n}\right\} \thicksim \left\{ y_{n}\right\} $ $%
\left( \text{reflexivity}\right) $ and $\left\{ y_{n}\right\} \thicksim
\left\{ x_{n}\right\} $ when $\left\{ x_{n}\right\} \thicksim \left\{
y_{n}\right\} $ $\left( \text{symmetry}\right) .$

Let $\left\{ x_{n}\right\} \thicksim \left\{ y_{n}\right\} $ and $\left\{
y_{n}\right\} \thicksim \left\{ z_{n}\right\} ,$ $z\in X$%
\begin{eqnarray*}
\lim_{n\rightarrow \infty }\left\Vert x_{n}-z_{n},z\right\Vert
&=&\lim_{n\rightarrow \infty }\left\Vert x_{n}-y_{n}+y_{n}-z_{n},z\right\Vert
\\
&\leq &\lim_{n\rightarrow \infty }K\left( \left\Vert
x_{n}-y_{n},z\right\Vert +\left\Vert y_{n}-z_{n},z\right\Vert \right) ,\text{
}K\geq 1 \\
&\leq &K\lim_{n\rightarrow \infty }\left\Vert x_{n}-y_{n},z\right\Vert
+\left\Vert y_{n}-z_{n},z\right\Vert \\
&=&K\left( 0+0\right) =0.
\end{eqnarray*}

Then $\left\{ x_{n}\right\} \thicksim \left\{ z_{n}\right\} $ $\left( \text{%
transitivity}\right) .$ So $\thicksim $ is a equivalence relation on $X.$
\end{proof}

\begin{theorem}
$\left\{ x_{n}\right\} $ is equivalent to $\left\{ a_{n}\right\} $ in a
linear $2$-normed space $\left( X,\left\Vert .,.\right\Vert \right) $ if and
only if 
\begin{equation*}
\lim_{n\rightarrow \infty }\left\Vert x_{n}-a_{n},z\right\Vert =0
\end{equation*}%
for every $z$ in $X$ (for proof, see\ [1]).
\end{theorem}

\begin{theorem}
If $\left\{ a_{n}\right\} $ and $\left\{ b_{n}\right\} $ are equivalent to $%
\left\{ x_{n}\right\} $ and $\left\{ y_{n}\right\} $ in a linear 2-normed
space $\left( X,\left\Vert .,.\right\Vert \right) ,$ respectively, then $%
\left\{ a_{n}+b_{n}\right\} $ is equivalent to $\left\{ x_{n}+y_{n}\right\} $
and $\left\{ \alpha a_{n}\right\} $ is equivalent to $\left\{ \alpha
x_{n}\right\} $ (for proof, see [1]).
\end{theorem}

Denote by $\widehat{X}$ the set of all equivalence classes of Cauchy
sequences in $X.$ Let $\widehat{x},$ $\widehat{y}$, $\widehat{z}$, etc.,
denote the elements of $\widehat{X}.$ Define an addition and scalar
multiplication on $\widehat{X}$ $\ $as follows:

\dag\ $\widehat{x}+\widehat{y}=$ the set of sequences of equivalent to $%
\left\{ x_{n}+y_{n}\right\} ,$ where $\left\{ x_{n}\right\} $ is in $%
\widehat{x}$ and $\left\{ y_{n}\right\} $ in $\widehat{y},$ and

\ddag $\alpha \widehat{x}=$ the set of sequences equivalent to $\left\{
\alpha x_{n}\right\} ,$ where $\left\{ x_{n}\right\} $ is in $\widehat{x}.$
It is clear that these two operations are well defined since they are
independent of the choice of elements from $\widehat{x}$ and $\widehat{y}.$
So $\widehat{X}$ is a linear space with operations.

\begin{theorem}
If $X$ is linear space having a uniformly continuous $2$-norm $\left\Vert
.,.\right\Vert $ defined on it, then for pairs of equivalent Cauchy
sequences and $\left\{ x_{n}\right\} \thicksim \left\{ a_{n}\right\} $ and $%
\left\{ y_{n}\right\} \thicksim \left\{ b_{n}\right\} $, Then%
\begin{equation*}
\lim_{n\rightarrow \infty }\left\Vert x_{n},y_{n}\right\Vert =\left\Vert
a_{n},b_{n}\right\Vert
\end{equation*}
\end{theorem}

\begin{theorem}
If $\left\{ x_{n}\right\} $ and $\left\{ y_{n}\right\} $ are Cauchy
sequences in a linear 2-normed space $\left( X,\left\Vert .,.\right\Vert
\right) $, then $\left\{ x_{n}-y_{n}\right\} $ is a Cauchy sequence in $X.$
\end{theorem}

\begin{proof}
We see that as follows:%
\begin{eqnarray*}
\left\Vert \left( x_{n}-y_{n}\right) -\left( x_{m}-y_{m}\right)
,z\right\Vert &=&\left\Vert \left( x_{n}-x_{m}\right) -\left(
y_{n}-y_{m}\right) ,z\right\Vert \\
&\leq &K\left( \left\Vert x_{n}-x_{m},z\right\Vert +\left\Vert
y_{n}-y_{m},z\right\Vert \right)
\end{eqnarray*}

we can readily see that, when $n\rightarrow \infty ,$ $\left\{
x_{n}-y_{n}\right\} $ is a Cauchy sequence in $X.$
\end{proof}

Whenever $X$ is a space having a uniformly continuous $2$-norm defined it
which is possible to define real-valued function on the space $\widehat{X}.$
The function is defined as follows:

For any two elements $\widehat{x}$ and $\widehat{y}$ in $\widehat{X},$%
\begin{equation*}
\left\Vert \widehat{x},\widehat{y}\right\Vert =\lim_{n\rightarrow \infty
}\left\Vert x_{n},y_{n}\right\Vert ,
\end{equation*}

Where $\left\{ x_{n}\right\} \in $ $\widehat{x}$ and $\left\{ y_{n}\right\}
\in \widehat{y}.$

Since the limit is exist and independent of the choice of the elements in $%
\widehat{x}$ and $\widehat{y}$. \ The function is well defined.

\begin{theorem}
If $X$ is a linear space having a uniformly continuous $2$-norm $\left\Vert
.,.\right\Vert $ defined on it and $\left\{ x_{n}\right\} $ and $\left\{
y_{n}\right\} $ are Cauchy sequences in $\widehat{x}$ and $\widehat{y}$,
respectively, then the function defined by 
\begin{equation*}
\left\Vert \widehat{x},\widehat{y}\right\Vert =\lim_{n\rightarrow \infty
}\left\Vert x_{n},y_{n}\right\Vert
\end{equation*}%
is a pseudo quasi $2$-norm on $\widehat{X}.$
\end{theorem}

\begin{proof}
By using definition of $2$-normed spaces, we see that,

$2N_{1})$ Let $\widehat{x}=\alpha \widehat{y}$%
\begin{eqnarray*}
\left\Vert \widehat{x},\widehat{y}\right\Vert &=&\left\Vert \alpha \widehat{y%
},\widehat{y}\right\Vert \\
&=&\lim_{n\rightarrow \infty }\left( \left\vert \alpha \right\vert
\left\Vert x_{n},y_{n}\right\Vert \right) \\
&=&0.
\end{eqnarray*}

$2N_{2})$ It is easy to see as follows:%
\begin{eqnarray*}
\left\Vert \widehat{x},\widehat{y}\right\Vert &=&\lim_{n\rightarrow \infty
}\left\Vert x_{n},y_{n}\right\Vert \\
&=&\lim_{n\rightarrow \infty }\left\Vert y_{n},x_{n}\right\Vert \\
&=&\left\Vert \widehat{y},\widehat{x}\right\Vert
\end{eqnarray*}

$2N_{3})$ On account of definition of $2$-normed space, that is,%
\begin{eqnarray*}
\left\Vert \alpha \widehat{x},\widehat{y}\right\Vert &=&\lim_{n\rightarrow
\infty }\left\Vert \alpha x_{n},y_{n}\right\Vert =\lim_{n\rightarrow \infty
}\left\vert \alpha \right\vert \left\Vert x_{n},y_{n}\right\Vert \\
&=&\left\vert \alpha \right\vert \lim_{n\rightarrow \infty }\left\Vert
x_{n},y_{n}\right\Vert \\
&=&\left\vert \alpha \right\vert \left\Vert \widehat{x},\widehat{y}%
\right\Vert .
\end{eqnarray*}

$2N_{4}^{\ast })$ $\widehat{x},$ $\widehat{y},$ $\widehat{z}\in \widehat{X},$
$\left\{ x_{n}\right\} \in \widehat{x},$ $\left\{ y_{n}\right\} \in \widehat{%
y},$ $\left\{ z_{n}\right\} \in \widehat{z}$ are Cauchy sequences%
\begin{eqnarray*}
\left\Vert \widehat{x},\widehat{y}+\widehat{z}\right\Vert
&=&\lim_{n\rightarrow \infty }\left\Vert x_{n},y_{n}+z_{n}\right\Vert \\
&\leq &\lim_{n\rightarrow \infty }\left( K\left( \left\Vert
x_{n},y_{n}\right\Vert +\left\Vert x_{n},z_{n}\right\Vert \right) \right) \\
&=&K\lim_{n\rightarrow \infty }\left\Vert x_{n},y_{n}\right\Vert
+K\lim_{n\rightarrow \infty }\left\Vert x_{n},z_{n}\right\Vert \\
&=&K\left( \left\Vert \widehat{x},\widehat{y}\right\Vert +\left\Vert 
\widehat{x},\widehat{z}\right\Vert \right)
\end{eqnarray*}

This shows that $\left\Vert \widehat{x},\widehat{y}\right\Vert
=\lim_{n\rightarrow \infty }\left\Vert x_{n},y_{n}\right\Vert $ is a pseudo
quasi $2$-norm on $\widehat{X}.$
\end{proof}

Let $\widehat{X}$ be the subset of $X$ consisting of those equivalence
classes which contain a Cauchy sequence $\left\{ x_{n}\right\} $ for which $%
x_{1}=x_{2}=...=x_{n}=\cdots $. At most one sequence of this kind can be in
each equivalence class. If $\widehat{x}$ and $\widehat{y}$ are in $\widehat{X%
}_{0}$ and if corresponding Cauchy sequence are $\left\{ x_{n}\right\} $ and 
$\left\{ y_{n}\right\} $ with $x_{n}=x$ and $y_{n}=y$ for every $n$, then we
have%
\begin{equation*}
\left\Vert \widehat{x},\widehat{y}\right\Vert =\lim_{n\rightarrow \infty
}\left\Vert x_{n},y_{n}\right\Vert =\left\Vert x,y\right\Vert .
\end{equation*}

Thus $\widehat{X}_{0}$ and $\widehat{X}$ are isometrics. This isometry will
be used to show that $\widehat{X}_{0}$ is dense in $\widehat{X}$ (for
details, see[1]).

\begin{theorem}
If $X$ is a linear space having a uniformly continuous quasi $2$-norm $%
\left\Vert .,.\right\Vert $ defined on it, $\overline{\left( \widehat{X}%
_{0}\right) }=\widehat{X}$ (for proof, see [1]).
\end{theorem}

\begin{theorem}
If $X$ is a linear space having a uniformly continuous quasi $2$-norm which
is defined as 
\begin{equation*}
\left\Vert \widehat{x},\widehat{y}\right\Vert =\lim_{n\rightarrow \infty
}\left\Vert x_{n},y_{n}\right\Vert
\end{equation*}%
then $\widehat{X}$ is complete and the pair $\left( \widehat{X},\left\Vert
.,.\right\Vert \right) $ $\ $is called completion of quasi $2$-normed space.
\end{theorem}

\begin{proof}
In order to see that $\left( \widehat{X},\left\Vert .,.\right\Vert \right) $
is complete, we have to show every Cauchy sequence in $\widehat{X}$ is
convergent in $\widehat{X}.$ Let $\left\{ a_{n}\right\} $\ be a Cauchy
sequence in $\widehat{X}$ and $\widehat{b}_{n}\in \widehat{X}_{0},$ $%
\widehat{c}_{n}\in \widehat{X}_{0}.$

Because of $\overline{\left( \widehat{X}_{0}\right) }=\widehat{X}$ then we
can write $\left\Vert \widehat{a}_{n}-\widehat{c}_{n},\widehat{b}\right\Vert
<\frac{1}{n}$ for each $n$, Also we have,%
\begin{eqnarray*}
\left\Vert \widehat{c}_{n}-\widehat{c}_{m},\widehat{b}\right\Vert
&=&\left\Vert \widehat{c}_{n}-\widehat{a}_{n}+\widehat{a}_{n}-\widehat{c}%
_{m},\widehat{b}\right\Vert \\
&\leq &K\left( \left\Vert \widehat{c}_{n}-\widehat{a}_{n},\widehat{b}%
\right\Vert +\left\Vert \widehat{a}_{n}-\widehat{c}_{m},\widehat{b}%
\right\Vert \right) \\
&=&K\left( \left\Vert \widehat{c}_{n}-\widehat{a}_{n},\widehat{b}\right\Vert
+\left\Vert \widehat{a}_{n}-\widehat{a}_{m}+\widehat{a}_{m}-\widehat{c}_{m},%
\widehat{b}\right\Vert \right) \\
&\leq &K\left\Vert \widehat{a}_{n}-\widehat{c}_{n},\widehat{b}\right\Vert
+K\left( K\left\Vert \widehat{a}_{n}-\widehat{a}_{m},\widehat{b}\right\Vert
+\left\Vert \widehat{a}_{m}-\widehat{c}_{m},\widehat{b}\right\Vert \right) ,%
\text{ for }K\geq 1 \\
&<&\frac{K}{n}+K^{2}\left\Vert \widehat{a}_{n}-\widehat{a}_{m},\widehat{b}%
\right\Vert +\frac{K^{2}}{m}
\end{eqnarray*}

Last from inequality when $n,m\rightarrow \infty $ right hand side will be
equal to $0.$ Thus 
\begin{equation*}
\lim_{n,m\rightarrow \infty }\left\Vert \widehat{c}_{n}-\widehat{c}_{m},%
\widehat{b}\right\Vert =0
\end{equation*}

this shows us that $\left\{ \widehat{c}_{n}\right\} $ is a Cauchy sequence
in $\widehat{X}.$

Use of $\widehat{X}$ and $\widehat{X}_{0}$ are isometric there is a Cauchy
sequence $\left\{ c_{n}\right\} $ in $X$ that corresponding $\left\{ 
\widehat{c}_{n}\right\} .$

On the other hand there is $\widehat{a}\in \widehat{X}$ such that $\widehat{a%
}\in \left\{ \widehat{c}_{n}\right\} $%
\begin{eqnarray*}
\left\Vert \widehat{a}_{n}-\widehat{a},\widehat{b}\right\Vert  &=&\left\Vert 
\widehat{a}_{n}-\widehat{c}_{n}+\widehat{c}_{n}-\widehat{a},\widehat{b}%
\right\Vert  \\
&\leq &K\left( \left\Vert \widehat{a}_{n}-\widehat{c}_{n},\widehat{b}%
\right\Vert +\left\Vert \widehat{c}_{n}-\widehat{a},\widehat{b}\right\Vert
\right)  \\
&<&\frac{K}{n}+K\left\Vert \widehat{c}_{n}-\widehat{a},\widehat{b}%
\right\Vert 
\end{eqnarray*}

Last from inequality as $n\rightarrow \infty $ and $\widehat{X}_{0}$ is
dense in $\widehat{X},$%
\begin{equation*}
\lim_{n\rightarrow \infty }\left\Vert \widehat{a}_{n}-\widehat{a},\widehat{b}%
\right\Vert =0
\end{equation*}

so arbitrary a Cauchy sequence $\left\{ \widehat{a}_{n}\right\} $ convergent
to $\widehat{a}\in \widehat{X}.$ Then $\left( \widehat{X},\left\Vert
.,.\right\Vert \right) $ is complete.
\end{proof}

\end{document}